\documentclass[11pt,reqno,a4paper]{amsart}

\usepackage{amsfonts}
\usepackage{epsfig}
\usepackage{graphicx}
\usepackage{amsmath}
\usepackage{amssymb}
\usepackage{color}
\usepackage{mathrsfs}

\newcommand{\R}{\mathbb{R}}

\newcommand{\Z}{\mathbb{Z}}

\newcommand{\nrm}[1]{\|#1\|}

\newcounter{main}

\numberwithin{equation}{section}

\newtheorem{theorem}{Theorem}[section]

\newtheorem{lemma}[theorem]{Lemma}

\newtheorem{maintheorem}{Theorem}

\newcommand{\blanksquare}{\,\,\,$\sqcup\!\!\!\!\sqcap$}

\newcounter{example}
{{\stepcounter{example}}{\flushleft {\bf Example \arabic{example}:}}}%
{\par}

\title[The $C^0$ general density theorem for geodesic flows]
{The $C^0$ general density theorem for geodesic flows}

\author[M. Bessa]{M\'{a}rio Bessa}

\address{Departamento de Matem\'atica, Universidade da Beira Interior, Rua Marqu\^es d'\'Avila e Bolama,
  6201-001 Covilh\~a,
Portugal.}
\email{bessa@ubi.pt}

\author[M. J. Torres]{Maria Joana Torres}
\address{Centro de Matem\'atica, Universidade do Minho e Departamento de Matem\'atica e Aplica\c{c}\~{o}es, Universidade do Minho,
Campus de Gualtar,
4700-057 Braga, Portugal}
\email{jtorres@math.uminho.pt}

\begin{document}

\begin{abstract}
Given a closed Riemannian manifold, we prove the $C^0$-general density theorem for continuous geodesic flows.
More precisely, that there exists a residual (in the $C^0$-sense) subset of the continuous geodesic flows such that, in that residual subset, the geodesic flow exhibits dense closed orbits.
\end{abstract}

\maketitle

{\tiny\noindent\emph{MSC 2000:} primary 28D05; 54H20; secondary 37B20.\\
\emph{Keywords:} Periodic points, topological dynamics, closing lemma.\\}

\begin{section}{Introduction: basic definitions and statement of the results}

\subsection{The geodesic flow and mechanic Hamiltonians}

A Riemannian manifold $(M,g)$ is a $C^\infty$-manifold with an Euclidean inner product $g_x$ in each $T_xM$ which varies smoothly with respect to $x \in M$.
So a Riemannian metric is a smooth section $g_M \rightarrow \mbox{\rm Symm}_2^+(TM)$, where $\mbox{\rm Symm}_2^+(TM)$ is the set of positive bilinear and symmetric
forms in $TM$. Pick a Riemannian metric on $TM$ and denote by $d_{TM}(\cdot,\cdot)$ the geodesic distance associated to it on $TM$.
Note that since all Riemannian metrics are Lipschitz equivalent on compact subsets, the choice of the metric on $TM$ is not important.

The geodesic flow of the metric $g$ of class $C^2$ is the flow on $TM$ defined by 
$$\begin{array}{cccl}
\phi^t_g: & T M & \longrightarrow & T M \\ \
{} & (x,v) & \longmapsto & (\gamma^g_{x,v}(t), \dot{\gamma}_{x,v}^g(t)), \\
\end{array}
$$
where, $\gamma_{x,v}^g: \R \longrightarrow M$ denotes the  geodesic starting at $x$ with initial velocity $v$, $x \in M$, $v \in T_x M$.
Since the speed of the geodesic is constant, we can consider the flow restricted to $UT(M):=\{(x,v) \in TM:\,g_x(v,v)=1\}$.  

Clearly, the orbit of a point $(x, v) \in  UT(M)$ consists of the
tangent vectors to the geodesic defined by $(x, v)$ and periodic orbits for the geodesic flow correspond to closed geodesics on $M$.
We write $Per(\phi^t_g) \subset UT(M)$ for the set of periodic orbits.
Recall that $(x,v) \in UT(M)$ belongs to the \emph{nonwandering set} of $\phi^t_g$, denoted by $\Omega(\phi^t_g)$, if for every neighborhood $V$ of $(x,v)$ there exists $t_n \rightarrow \infty$ such that $\phi^{t_n}_g(V) \cap V \neq \emptyset$. We say that $(x,v) \in UT(M)$ is a $\phi^t_g$-\emph{recurrent point}, and we denote this set by $R(\phi^t_g)$, if given any neighborhood $V$ of $(x,v)$, there exists $t_n$ such that $\phi^{t_n}_g((x,v))\in V$. We have $R(\phi^t_g)\subset \Omega(\phi^t_g)$. Along this paper we are going to consider that $M$ is closed and with dimension $\geq 2$. We observe that the geodesic flow keeps the Liouville volume invariant. Thus, Poincar\'e's recurrence theorem (see e.g. ~\cite{KH}) asserts that Lebesgue almost every point is recurrent. Therefore, we conclude that Lebesgue almost every point is nonwandering and so $UT(M)=\Omega(\phi^t_g)$.

The geodesic flow is a subclass of the Hamiltonian flows. Actually, for the metric $g$ on $M$ we have the associated Hamiltonian defined in the cotangent bundle $T^*M$ endowed with the canonical symplectic form by:
$$\begin{array}{cccl}
H\colon & T^*M & \longrightarrow & \mathbb{R} \\ \
{} & (x,p) & \longmapsto & \frac{1}{2}(\| p \|_x^*)^2, \\
\end{array}
$$where $\|\cdot\|^*$ stands for the dual norm on the cotangent bundle. Usually, we call $H$ the \emph{mechanical Hamiltonian} because it is given only by the kinetic energy and free of potential.

\subsection{Statement of the result and some history}

 Let $\mathcal{R}^k(M)$ denote the set of $C^k$ Rie\-manni\-an metrics in $M$. For $\ell\in\{1,...,k\}$ we can endow $\mathcal{R}^k(M)$ with the $C^\ell$-topology. A property in $\mathcal{R}^k(M)$, endowed with the $C^\ell$ topology  is said to be $C^\ell$-\emph{generic} if it holds in a $C^\ell$-residual subset of $\mathcal{R}^k(M)$.
In particular, since when $\ell=k$ the set $\mathcal{R}^k(M)$ is a Baire space, by Baire's Category theorem (see~\cite{Ku}) a $C^k$-residual subset of $\mathcal{R}^k(M)$ is $C^k$-dense in $\mathcal{R}^k(M)$. 
Consequently, a $C^k$-residual subset of $\mathcal{R}^k(M)$ is $C^\ell$-dense in $\mathcal{R}^k(M)$, for any $\ell\in\{1,...,k\}$.

Let $C^0([a,b]\times UT(M),UT(M))$ stands for the set of continuous maps from $[a,b]\times UT(M)$ into $UT(M)$. We say that $\phi$ is a $C^0$ flow on the unit tangent bundle $UT(M)$ if $\phi\colon \mathbb{R}\times UT(M)\rightarrow UT(M)$ is such that $\phi(t,\cdot)$ is a homeomorphism and $\phi(\cdot, (x,v))$ is a $C^1$ curve on $UT(M)$. As usually we denote a flow by $\phi^t$. We are interested in volume-preserving flows, i.e., the Lebesgue measure is $\phi^t$-invariant. We let $\mathscr{F}^0(UT(M))$ be the set of volume-preserving $C^0$ flows on $UT(M)$. For any $a<b$ we define the map:
$$\begin{array}{cccc}
\rho_{a,b}: & \mathscr{F}^0(UT(M)) & \longrightarrow & C^0([a,b]\times UT(M),UT(M)) \\
& \phi^t & \mapsto & \phi^t|_{[a,b]\times UT(M)}
\end{array}
$$

The compact-open topology, that we shall denote by $\tau$, is the smallest one making $\rho_{a,b}$ continuous (see \cite{PR,Hi}). We endow $ \mathscr{F}^0(UT(M))$ with $\tau$. Let $\mathscr{G}^1(UT(M))$ be the set of geodesic flows associated to metrics in $\mathcal{R}^2(M)$. We define the set of \emph{continuous geodesic flows} by the $\tau$-closure of $\mathscr{G}^1(UT(M))$ and denote this set by $\mathscr{G}^0(UT(M))\subset \mathscr{F}^0(UT(M))$. Observe that $(\mathscr{G}^0(UT(M)), \tau)$ is Baire (\cite[\S2.4]{Hi}). Moreover, if $\phi^t_{g_1}, \phi^t_{g_2}\in \mathscr{G}^1(UT(M))$ and $g_1$ and $g_2$ are $C^1$-close, then as a consequence of Gronwall's inequality (see e.g. ~\cite{PM}) we get that $\phi^t_{g_1}$ and $\phi^t_{g_1}$ are $\tau$-close. A property in $\mathscr{F}^0(UT(M))$ is said to be $\tau$-\emph{generic} if it holds in a $\tau$-residual subset of $\mathscr{F}^0(UT(M))$. Once again, by Baire's Category theorem, a $\tau$-residual subset of $\mathscr{F}^0(UT(M))$ is $\tau$-dense in $\mathscr{F}^0(UT(M))$.

Given any continuous geodesic flow $\phi \in \mathscr{G}^0(UT(M))$, the definitions of the nonwandering and recurrent sets of $\phi^t$ can be given analogously to the ones given for a geodesic flow $\phi_g^t \in \mathscr{G}^1(UT(M))$.

Considering $C^0$-closures of geodesic flows is a well-known subject and it is quite related to Gromov-Eliasberg symplectic rigidity (see e.g. \cite{OM,O,V}). We mention a recent result on rigidity of $C^0$-geodesic flows (\cite[\S6]{MS}) which, in rough terms says that, if the sequence of metrics $g_n\in \mathcal{R}^2(M)$ converges in the $C^0$-sense to $g$, then the geodesic flows associated to $g_n$, $\phi^t_{g_n}$, converge to the geodesic flow of $g$.

Given a Riemannian metric $g$ which generates a geodesic flow $\phi^t_g$, a central question in dynamical systems is to know if the periodic orbits of $\phi^t_g$ are dense in $\Omega(\phi^t_g)$. The aim of the present paper is to prove the celebrated Pugh's general density theorem (\cite{Pu}) for continuous geodesic flows, i.e. to show that:

\begin{maintheorem}\label{Teo1}
There exists a $\tau$-residual subset $\mathscr{G}$ of $\mathscr{G}^0(UT(M))$ such that $UT(M) = \overline{Per(\phi^t)}$, for any $\phi^t \in \mathscr{G}$.
\end{maintheorem}

We recall that Klingenberg and Takens theorem (see ~\cite{KT}) assures that for a dense subset of metrics on a compact manifold $M$ we have infinitely many closed geodesics (see also \cite{Rad} on more general generic assumptions). It is worth noting that on manifolds with negative curvature the geodesic flow is Anosov and so, by Anosov closing lemma (see e.g. \cite{KH}), the closed geodesics are dense in the manifold without any need of generic considerations.

The first attempt to obtain the dissipative version of Theorem~\ref{Teo1} in the much less knotty class of homeomorphisms was in \cite{PPSS} by Palis, Pugh, Shub and Sullivan. However, the proof in \cite{PPSS} was not complete as it was observed in ~\cite{CMN}. The complete proof is due to Hurley (\cite{H}) using dissipative arguments. Hurley's proof uses the fact that we can create a $C^0$-stable periodic \emph{sink} by a small $C^0$-perturbation and Brouwer's fixed point theorem guarantees a fixed point for every $C^0$-close homeomorphism. Clearly, this strategy is meaningless for the volume-preserving setting because sinks simply do not exist. With respect to the volume-preserving context Daalderop and Fokkink proved (see ~\cite[Proposition 4]{DF}) that the general density theorem holds for that systems. In the present paper we mainly follow the arguments in~\cite{CMN}, adapting them to the geodesic flow framework and combine it with the recent proved closing lemma (\cite{Ri}). It is interesting to remark that the general density theorem was first established in the $C^1$-class by Pugh (\cite{Pu}) in the late 1960's and much later by Pugh and Robinson for volume-preserving and symplectic diffeomorphisms as also for Hamiltonians (see~\cite{PR}). Actually, the general density  theorem is a direct consequence of the combination of the closing lemma and the stability and persistence of non-degenerated closed orbits given, for example, from the hyperbolic (and also elliptic in the conservative case) structure. This stability and persistence holds in the smooth case but sadly the notion of hyperbolicity (and ellipticity) is no longer valid in the topological context, as it is perceptively observed by Rifford in \cite{Ri}:
\emph{``The Pugh $C^1$-Closing Lemma has strong consequences on the structure of the flow of generic vector fields... It is worth noticing that our result is not striking enough to infer relevant properties for generic geodesic flows (for instance, the existence of an hyperbolic periodic orbit is not stable under $C^0$ perturbations on the dynamics)"}. Fortunately, we have at hand a range of topological techniques which will allow us to reach a generic result of undeniable interest.

\end{section}

\begin{section}{Perturbation results}

Let $\phi^t_g$ be the geodesic flow of a Riemannian metric $g\in \mathcal{R}^2(M)$ acting on $UT(M)$, the unit tangent bundle of $M$. Let $\pi:UT(M) \rightarrow M$ be the canonical projection.
Non-trivial closed geodesics on $M$ are in one-to-one correspondence
to the periodic orbits of  $\phi^t_g$. From now on we denote by $\theta_\gamma$ the representative of the closed geodesic $\gamma$ by picking a single point in the orbit, say $\theta_\gamma=(x,v)$.
 Given a closed orbit $\gamma =\{ \phi^t_g(\theta): \, t \in [0,a]\}$ of period $a$ we can define the {\em Poincar\'{e} map}
$\mathcal{P}_g(\Sigma,\gamma)$ as follows: one can choose a local $(2\dim(M)-2)$-hypersurface $\Sigma$ in $UT(M)$ containing $\theta$ and transversal to $\gamma$ such that there are open neighborhoods
$\Sigma_0$ and $\Sigma_a$ of $\theta$ in $\Sigma$ and a differentiable arrival function $\delta: \Sigma_0 \rightarrow \R$
with $\delta(\theta)=a$
such that the map $\mathcal{P}_g(\Sigma,\gamma): \Sigma_0 \rightarrow \Sigma_a$ given by $v \mapsto \phi_g^{\delta(v)}(v)$ is a diffeomorphism.

Given a closed geodesic $c:\R/\Z \rightarrow M$, all iterates $c^m: \R / \Z \rightarrow M$; $c^m(t)=c(mt)$ for a positive integer $m$ are closed geodesics too.

A closed orbit $\gamma$ (or the corresponding closed geodesic $c$) is called {\em nondegenerate} (c.f.~\cite{CP}) if $1$ is not an eigenvalue of the linearized Poincar\'{e}
map $P_c:=D_{\gamma(0)}\mathcal{P}_g(\Sigma,\gamma)$. In that case, $\gamma$ is an isolated closed orbit and $\pi \circ \gamma$ is an isolated closed geodesic.

A Riemannian metric $g$ is called {\em bumpy} if all the closed orbits of the geodesic flow are nondegenerate.
Since $P_{c^m}=P^m_c$ this is equivalent to saying that if $\mbox{\rm exp}(2 \pi i \lambda)$ is an eigenvalue of $P_c$, then $\lambda$ is irrational.
We state the bumpy metric theorem~(\cite{Ab,An}):

\begin{theorem}\label{bumpy}(Bumpy metric theorem)
For $2 \leq k \leq \infty$, the set of bumpy metrics of class $C^k$ is a residual subset of $\mathcal{R}^k(M)$.
\end{theorem}

Recently, in~\cite{Ri}, Rifford was able to overcome a problem that was open for a long time and showed how to close an orbit of the geodesic flow by a small perturbation of the metric in the $C^1$ topology:

\begin{theorem}\label{closing}($C^1$-closing lemma)
Let $g$ be a Riemannian metric on $M$ of class $C^k$ with $k \geq 3$ (resp. $k=\infty$), $(x,v) \in UT(M)$ and $\epsilon >0$ be fixed. Then there exists a metric $\tilde{g}$ of class $C^{k-1}$ (resp. $C^{\infty}$) with $\nrm{\tilde{g} - g}_{C^1} < \epsilon$ such that the geodesic $\gamma_{(x,v)}^{\tilde{g}}$ is periodic. 
\end{theorem}

\end{section}

\begin{section}{Proof of Theorem~\ref{Teo1}}

Let $\phi^t \in \mathscr{G}^0(UT(M))$. 
We define the set of \emph{weak-periodic points of $\phi^t$} by

$$WPer(\phi^t):=  \{(x,v)=\underset{n\rightarrow\infty}{\lim} (x_n,v_n) \colon 
 (x_n,v_n)\in Per(\phi_{n}^t), \, \phi_{n}^t \in \mathscr{G}^0(UT(M))\text{ and }
\phi_n^t \underset{\tau}{\to} \phi^t \}$$

We observe that, thanks to the Poincar\'{e} recurrence theorem, the geodesic flow $\phi^t$ is nowandering on $UT(M)$.
Let $\mathscr{G}^2(UT(M))$ be the set of geodesic flows associated to metrics in $\mathcal{R}^3(M)$. Clearly,
$\mathscr{G}^2(UT(M))$ is $\tau$-dense in $\mathscr{G}^0(UT(M))$. 
Therefore, 
it follows from the $C^1$-closing lemma (Theorem~\ref{closing}), and via Gronwall's inequality, that:

\begin{lemma}\label{weak}
If $\phi^t \in \mathscr{G}^0(UT(M))$, then $UT(M) = WPer(\phi^t)$.
\end{lemma}

We say that a closed orbit $\gamma$ of $\phi^t$ is \emph{permanent} if any ${\tilde{\phi}}^t \in \mathscr{G}^0(UT(M))$  and $\tau$-arbitrarily close to $\phi^t$ has a 
${\tilde{\phi}}^t$-periodic orbit $\tilde{\gamma}$ near $\gamma$. Let $\mathscr{P}(\phi^t)$ denote the set of all permanent closed orbits of $\phi^t$.

\begin{lemma}\label{perm}
There exists a $\tau$-residual subset $\mathscr{G}$ of $\mathscr{G}^0(UT(M))$ such that $Per(\phi^t)=\mathscr{P}(\phi^t)$,
for any $\phi^t\in \mathscr{G}$.
\end{lemma}

\begin{proof} Along this proof we borrow some argument developed in ~\cite{CMN} together with some elementary fixed point index theory.

It follows from the bumpy metric theorem (Theorem~\ref{bumpy}) that there exists a $C^2$-residual subset $\mathscr{R}_0$ of $\mathcal{R}^2(M)$, hence $\mathscr{R}_0$ is $C^1$-dense in $\mathcal{R}^2(M)$, such that
every metric $g$ in $\mathscr{R}_0$ has all the closed orbits of the geodesic flow nondegenerate. As a consequence of Gronwall's inequality, there exists a $\tau$-dense subset $\mathscr{G}_0$ of $\mathscr{G}^1(UT(M))$ such that every $\phi_g^t$ in $\mathscr{G}_0$ has all the closed orbits nondegenerate. Since 
$\mathscr{G}^1(UT(M))$ is $\tau$-dense in $\mathscr{G}^0(UT(M))$ we get that $\mathscr{G}_0$ is $\tau$-dense is $\mathscr{G}^0(UT(M))$.

We claim that there exists a $\tau$-residual $\mathscr{G}$ such that any element in it has all the closed orbits permanent. For that we take a countable base for the topology $\{\mathcal{B}_i\}_{i\in\mathbb{N}}$ of the 
unit tangent bundle $UT(M)$. 
The fixed point index will play a crucial role along  the proof since, in rough terms the existence of non-null index on a set assures a fixed point (periodic orbit) in that set and, moreover,  displaying non-null index persists under topological $\tau$-perturbations.

Now we define, for every $i,n\in\mathbb{N}$,  the following $\tau$-open subsets of $\mathscr{G}^0(UT(M))$, which a priori do not cover the whole set $\mathscr{G}^0(UT(M))$, in the following way:
\begin{enumerate}
\item 
$\phi^t\in \mathcal{F}_{i,n}$ if $\phi^t$ is \emph{free} of closed orbits of period less or equal than $n$ in $\overline{\mathcal{B}_i}$;
\item $\phi^t\in \mathcal{I}_{i,n}$ if there exists $\mathcal{B}_j$ with $\text{diam}(\mathcal{B}_j)<\text{diam}(\mathcal{B}_i)$ such that for some return time $t<n$ we have $\phi^t(x,v)\not= (x,v)$ for any $(x,v)$ in the boundary of $\mathcal{B}_j$ and $ind(\phi^t,\mathcal{B}_j)\not=0$.
\end{enumerate}

Let $\mathscr{O}^{\mathcal{R}}_n$ be the $C^1$-open and dense subset of metrics in $\mathcal{R}^2(M)$ without any closed geodesic of period exactly equal to $n$, and let $\mathscr{O}_n$ be the set of geodesic flows associated with metrics in $\mathscr{O}^{\mathcal{R}}_n$.

We claim that any flow $\phi^t_g\in\mathscr{G}_0\cap \mathscr{O}_n$ is such that $\phi^t_g\in \mathcal{F}_{i,n}\cup \mathcal{I}_{i,n}$ for all $i,n$. In fact, if $\phi^t_g$ has no closed orbit  with period $\leq n$ through $\mathcal{B}_i$, then $\phi^t_g\in \mathcal{F}_{i,n}$, otherwise $\phi^t_g$ has a closed orbit $\gamma$ with period $a<n$ through $\mathcal{B}_i$. But, since $\phi^t_g\in\mathscr{G}_0$, $\gamma$ must be isolated from other closed orbits with period less than $a$. Therefore, we accomplish (2) and $\phi^t_g\in \mathcal{I}_{i,n}$. 

We let the residual on the theorem to be defined by
$$\mathscr{G}:=\bigcap_{i,n} (\mathcal{F}_{i,n}\cup \mathcal{I}_{i,n}).$$

Let us see why it works; take $\phi^t\in\mathscr{G}$ and $\gamma\in Per(\phi^t)$ of period $a$, then for any $\mathcal{B}_i$ which intersects $\gamma$ clearly $\phi^t\in \mathcal{I}_{i,n}$ for all $n\geq a$.  As a consequence, there exists $\mathcal{B}_j$ with $\text{diam}(\mathcal{B}_j)<\text{diam}(\mathcal{B}_i)$ and such that $ind(\phi^a,\mathcal{B}_j)\not=0$ and moreover this property is persistent for small $\tau$-perturbations on the original flow. In conclusion $\gamma$ is permanent.

\end{proof}

Let $UT(M)^\star$ be the set of compact subsets of $UT(M)$ endowed with the Hausdorff topology.

\begin{lemma}\label{lsc}
The map $\mathfrak{P}\colon \mathscr{G}^0(UT(M))\rightarrow UT(M)^\star$, where $\mathscr{G}^0(UT(M))$ is endowed with the $\tau$-topology and $UT(M)^\star$ is endowed with the Hausdorff topology, and defined by $\mathfrak{P}(\phi^t)=\overline{Per(\phi^t)}$ is lower semicontinuous on the residual $\mathscr{G}$ given by Lemma~\ref{perm}.
\end{lemma}

\begin{proof}
We must prove that for any $\tilde{\phi}^t\in\mathscr{G}$, and any $\epsilon>0$ there exists a neighborhood $V$ of $\tilde{\phi}^t$ such that $\mathfrak{P}(\tilde{\phi}^t) \subseteq \mathcal{B}_{\epsilon}(\mathfrak{P}(\phi^t))$ for all $\phi^t\in V$, or in other words there are no implosions of the number of closed orbits when we perturb $\tilde{\phi}^t$. But Lemma~\ref{perm} says that $Per(\tilde{\phi}^t)=\mathscr{P}(\tilde{\phi}^t)$ and the proof is completed by recalling the definition of permanent closed orbit.
\end{proof}

\medskip

\begin{proof}(of Theorem~\ref{Teo1})
Noting that $UT(M)=\Omega(\phi^t)$ we will prove that $\Omega(\phi^t)=\overline{Per(\phi^t)}$ for some $\tau$-generic flow in $\mathscr{G}^0(UT(M))$. From Lemma~\ref{lsc} the map 
$\mathfrak{P}\colon\mathscr{G}^0(UT(M))\rightarrow UT(M)^\star$
 defined by $\mathfrak{P}(\phi^t)=\overline{Per(\phi^t)}$ is lower semicontinuous on $\mathscr{G}$. It is well-known (see~\cite{F}) that the continuity points of $\mathfrak{P}|_{\mathscr{G}}$ is a residual subset $\mathscr{G}_1\subset \mathscr{G}$, hence a residual subset of $\mathscr{G}^0(UT(M))$. Let us see that if $\phi^t\in\mathscr{G}_1$ (i.e. $\phi^t$ is a continuity point of $\mathfrak{P}|_{\mathscr{G}}$) then $UT(M)=\overline{Per(\phi^t)}$. The non obvious inclusion is $UT(M)\subset\overline{Per(\phi^t)}$. Assume, by contradiction, that exists $(x,v) \in UT(M) \setminus \overline{Per(\phi^t)}$. 
By Lemma~\ref{weak} let $\{\phi_n^t\}_{n\in\mathbb{N}}\subset\mathscr{G}^0(UT(M))$ and $\{(x_n,v_n)\}_{n\in\mathbb{N}} \subset Per(\phi_n^t)$ be such that $\phi_n^t \underset{\tau}{\to} \phi^t$
and $\underset{n\rightarrow\infty}{\lim} (x_n,v_n)=(x,v)$. 
 By Lemma~\ref{perm} there exists $\{\tilde{\phi}_n^t\}_{n\in\mathbb{N}}$ such that $(x_n,v_n)\in\mathscr{P}({\tilde{\phi}_n}^t)$ and  each $\tilde{\phi}_n^t$ becomes $\tau$-arbitrarilly close to $\phi_n^t$.

Since $\mathscr{G}^0(UT(M))$ endowed with the $\tau$-topology is a Baire space we get that $\mathscr{G}$ is dense in $\mathscr{G}^0(UT(M))$. 
Therefore, there exist 
$\{\overline{\phi}_n^t\}_{n\in\mathbb{N}} \subset \mathscr{G}$ and 
$\{(\overline{x}_n,\overline{v}_n)\}_{n\in\mathbb{N}} \subset Per(\overline{\phi}_n^t)$ 
such that
$$d_{UT(M)}((x_n,v_n),(\overline{x}_n,\overline{v}_n))<\frac{1}{n}$$ and $\overline{\phi}_n^t$ becomes $\tau$-arbitrarilly close to 
$\tilde{\phi}_n^t$.

We conclude that 
$ \overline{\phi}_n^t \underset{\tau}{\to} \phi^t$
 and $\underset{n\rightarrow\infty}{\lim} (\overline{x}_n,\overline{v}_n)=(x,v)$.
Then, since $\phi^t$ is a continuity point of $ \mathfrak{P}|_{\mathscr{G}}$ we have that $\underset{n\rightarrow\infty}{\lim}\mathfrak{P}(\overline{\phi}_n^t)=\mathfrak{P}(\phi^t)$, i.e., $\underset{n\rightarrow\infty}{\lim}\overline{Per(\overline{\phi}_n^t)}=\overline{Per(\phi^t)}$. Finally, we observe that
$(x,v)\in \overline{Per(\overline{\phi}_n^t)}$, or equivalently,
$(x,v)\in \overline{Per(\phi^t)}$ which is a contradiction with the assumption that $(x,v)\in UT(M)\setminus \overline{Per(\phi^t)}$.

\end{proof}

\end{section}

\section*{Acknowledgements}

MB was partially supported by National Funds through FCT - ``Funda\c{c}\~{a}o para a Ci\^{e}ncia e a Tecnologia'', project PEst-OE/MAT/UI0212/2011. MJT was partially supported by the Research Centre of Mathematics of the University of Minho 
with the Portuguese Funds from the ``Funda\c{c}\~{a}o para a Ci\^{e}ncia e a Tecnologia", through the Project PEstOE/MAT/UI0013/2014.


\end{document}